\documentclass[a4paper,11pt,twoside]{article}
\usepackage{pst-fill,pst-grad}
\usepackage{textcomp}
\usepackage[english]{babel}
\usepackage{graphicx}
\usepackage{amsmath}
\usepackage{float}
\usepackage[matrix,arrow,curve]{xy}
\usepackage{pstricks}
\usepackage{amsmath,amsfonts,verbatim,afterpage,theorem,euscript,mathrsfs,amssymb}
\usepackage{amsfonts}
\usepackage{amssymb}
\usepackage{eulervm}
\usepackage{array}
\usepackage{dsfont}
\newtheorem{Theoreme}{Theorem}

\newtheorem{Definition}{Definition}[section]
\newtheorem{Proposition}{Proposition}[section]
\newtheorem{Lemme}{Lemma}[section]

\newtheorem{Remarque}{Remark}[section]

\numberwithin{equation}{section}
\newenvironment{proof}{\paragraph{Proof:}}{\hfill$\square$}
\def\vu{\vec{u}}
\def\vv{\vec{v}}
\def\vf{\vec{f}}
\def\vg{\vec{g}}

\def\vn{\vec{\nabla}}

\def\Rt{\mathbb{R}^2}
\def\R{\mathbb{R}}

\def\mPl{\mathcal{P}^{\text{log}}}

\title{\bf Remarks on variable Lebesgue spaces and  fractional Navier-Stokes equations}
\usepackage{authblk}
\author{Gast\'on Vergara-Hermosilla\footnote{\emph{coibungo@gmail.com}} }
\affil{\footnotesize LaMME, Univ. Evry, CNRS, Universit\'e Paris-Saclay, 91025, Evry, France.}
\begin{document}

	\maketitle

	\begin{abstract}  
In this work we study the 3D Navier-Stokes equations, under the action of an external force and with the fractional Laplacian operator $(-\Delta)^{\alpha}$ in the diffusion term, from the point of view of variable Lebesgue spaces. Based on decay estimates of the fractional heat kernel we prove the existence and uniqueness of mild solutions on this functional setting. Thus, in a first theorem we obtain an unique local-in-time solution in the space $L^{p(\cdot)} \left( [0,T], L^{q} (\R^3) \right)$. As a bi-product, in a second theorem we prove the existence of an unique global-in-time solution in the mixed-space $\mathcal{L}^{p(\cdot)}_{\frac{3}{2\alpha -1}}(\mathbb{R}^3,L^\infty([0,T[))$.	
\end{abstract}

\tableofcontents    
\section{Introduction}

\subsection{General setting}
In this paper we consider the 
fractional incompressible Navier-Stokes equations defined in $\R^3$, namely the system 
\begin{equation}\label{NS_Intro}
\begin{cases}
\partial_t\vu=- (-\Delta)^{\alpha} \vu-(\vu \cdot \vn) \vu+\vn P + \vec{f}, \qquad \\[3pt]
div(\vu)=0,\\[3pt]
\vu(0,x)=\vu_0(x),\quad div(\vu_0)=0, \qquad x\in \mathbb{R}^3,
\end{cases}
\end{equation}
with $\alpha\in ]1/2,1]$ and $(-\Delta)^{\alpha}$
is the fractional Laplacian operator, which is 
defined at the Fourier level by the symbol $|\xi|^{2\alpha}$. 
Considering the traditional notation, the vector field 
$\vu:[0, +\infty[\times \mathbb{R}^3 \longrightarrow \mathbb{R}^3$ denotes the velocity of a viscous, incompressible and homogeneous fluid,   $P:[0, +\infty[\times\mathbb{R}^3 \longrightarrow  \mathbb{R}$ is its pressure and $\vu_0:\mathbb{R}^3 \longrightarrow \mathbb{R}^3$, $\vf:[0, +\infty[\times \mathbb{R}^3 \longrightarrow \mathbb{R}^3$ stand for a given initial data and a given external force, respectively. \\

This system is of interest by several reasons. For instance, it can be seen as a natural generalization of the classical incompressible 3D Navier-Stokes equations. 
In fact, equations \eqref{NS_Intro} are the equations resulting from replacing the Laplacian $(-\Delta)$ in the Navier-Stokes equation by $(-\Delta)^{\alpha}$.
On the other hand, system \eqref{NS_Intro} has similar energy estimates and scaling properties as the classical Navier-Stokes equations.
In this case, the existence and uniqueness of mild solutions was pioneer by the works of Fujita and Kato \cite{FujitaKato64,katofujita}, where they reformulated the classical Navier-Stokes equation into an integral equation and proved the local well-posedness 
on appropriate Sobolev and Lebesgue spaces. Since \cite{FujitaKato64,katofujita}, many other functional spaces had been considered in the literature, such as Besov spaces \cite{Cannoneonde,PG02}, Morrey spaces \cite{GigaMikawa,Kato92}, Fourier-Herz spaces \cite{CANNONEWU} and the $BMO^{-1}$ space \cite{KT01}. 
For a rigorous review about these and others possible functional spaces, see the book \cite{Lemarie2016}. \\

In this work we are concerned by the existence of \emph{mild} solutions of \eqref{NS_Intro} which are obtained via a fixed-point method. 
To develop this theory, the choice of a \emph{appropriate} functional setting is crucial. 
Lions in \cite{MR0244606} proved that in the case that $\alpha \geq 5/4$, the  fractional 3D Navier-Stokes equations \eqref{NS_Intro} have a global and unique regular solution on appropriate Lebesgue and Sobolev spaces (for a simple proof of this fact, see the appendix of \cite{MR2016814}). Thus, since the seminal work of Lions  many other functional spaces has seen considered to obtain global solutions for \eqref{NS_Intro}, for instance Q-type spaces \cite{MR2679016,MR2651705}, Besov spaces \cite{wu2004generalized,yu2012well}, Triebel-Lizorki type spaces \cite{MR3580298}, the $BMO^{1-2\alpha}$ space \cite{MR2656417}, etc.
\\


In this paper we focus our efforts on the exploration of some existence and uniqueness results for equations (\ref{NS_Intro}) 
considering as functional framework the Lebesgue spaces of variable exponent $L^{p(\cdot)}(\R^3)$.
Intuitively, these spaces are a natural generalization of the classical Lebesgue spaces $L^p$, however, in the construction of them there exist particular issues that make these spaces quite dissimilar. 
Being more precises, in our functional setting the (classical) constant parameter $p\in[1,+\infty[$ is replaced by a measurable function $p(\cdot):\mathbb{R}^3\longrightarrow [1,+\infty[$.
Thus, in order to construct the spaces  $L^{p(\cdot)}$, we start by considering a measurable function $\vf:\mathbb{R}^3\longrightarrow \mathbb{R}^3$, and we define the \emph{modular function} $\varrho_{p(\cdot)}$ associated to $p(\cdot)$ by the formula
\begin{equation}\label{Def_Modular_Intro}
\varrho_{p(\cdot)}(\vf)=\int_{\mathbb{R}^3}|\vf(x)|^{p(x)}dx.
\end{equation}
In this point,
we should remark that in the case that $p(\cdot)\equiv p\in [1,+\infty[$, it is possible recover the classical Lebesgue spaces $L^p$. In fact, the $L^p$-norm can be obtained by considering 
the modular function as follow, 
$$\|\vf\|_{L^{p}(\R^3)}=\left(
 \varrho_{p}(\vf)
 \right)^{\frac1p}.$$
At this stage a relevant difference arises; if we consider the measurable function  $p(\cdot)$, we cannot simply consider the variable exponent $\frac{1}{p(\cdot)}$
instead of the constant exponent $\frac1p$ in the previous expression. 
In order to by-pass this issue, it is classical to equip the space $L^{p(\cdot)}$
with the  \emph{Luxemburg norm} associated to the modular function $\varrho_{p(\cdot)  }$ (see \cite{Cruz_Libro},  \cite{Diening_Libro} and \cite{BookALEX}), which is given by the expression 
\begin{equation}\label{Def_LuxNormLebesgue}
\|\vf\|_{L^{p(\cdot)} (\R^3) }=\inf\left\{\lambda > 0: \, 
\int_{\mathbb{R}^3}\left|\frac{\vf(x)}{\lambda}  \right|^{p(x)}dx 
\leq1 \right\}.
\end{equation}
Then, we define the variable Lebesgue spaces
$L^{p(\cdot)}(\mathbb{R}^3)$ as the set of measurable functions such that the quantity $\|\cdot\|_{L^{p(\cdot)}} (\R^3)$ above is finite.
A comprehensive presentation on the theory of variable Lebesgue spaces can be consulted in the books  \cite{Cruz_Libro},  \cite{Diening_Libro} and \cite{BookALEX}. \\

To the best of our knowledge, these functional spaces has not been considered before in the analysis of the fractional Navier-Stokes equations. Thus, our aim in this article is to present a first application of the
variable Lebesgue spaces
to the study of this system. 
More precisely, here we study 
mild solutions for the fractional Navier-Stokes equations considering  variable Lebesgue spaces $ L^{p(\cdot)} (\R^3) $ as functional setting. 
The presentation of these results motivate the next subsection.

\subsection{Presentation of the results} 

In this subsection we state our main results about the well-posedness of the 3D fractional Navier-Stokes equations on  $ L^{p(\cdot)}$-spaces. 
Thus, in order to precise the presentation of our theorems, we must introduce the following two definitions. In the first of them, we formalize 
the notion of variable exponent,  set (or class) of variable exponents and limit exponents.
\begin{Definition}
	Let consider $n\in \mathbb{N}$, and a function $p:\mathbb{R}^n\longrightarrow [1,+\infty[$. Then, $p$ will be called a {\it variable exponent} if $p(\cdot)$ is a measurable function and we define $\mathcal{P}(\mathbb{R}^n)$ to be the set of variable exponents. Moreover, we define the {\it limit exponents} $p^-= {\mbox{inf ess}}_{x\in \mathbb{R}^n}  \; \{p(x)\}$ and $p^+= {\mbox{sup ess}}_{x\in \mathbb{R}^n} \; \{p(x)\}$. 
\end{Definition}
In the next, we define a class of variable exponents which is related with the  boundedness  of some classical operators involved in the analysis of Navier-Stokes type equations, such as the Hardy-Littlewood maximal function and Riesz potentials (see Subsection \ref{Secc_Notaciones_Presentaciones} for precise details about it).

\begin{Definition}
	Let $p(\cdot)\in \mathcal{P}(\mathbb{R}^n)$ and consider the limit value 
	$
	\frac{1}{p_\infty}=\underset{|x|\to +\infty}{\lim}\frac{1}{p(x)}.$ 
	We say that $p(\cdot)\in \mathcal{P}(\mathbb{R}^n)$ belongs to the class
	$\mathcal{P}^{log}(\mathbb{R}^n)$
	if we have:  
	\begin{enumerate}
		\item $\left|\frac{1}{p(x)}-\frac{1}{p(y)}\right|\leq \frac{C}{\log(e+1/|x-y|)}$ for all $x,y\in \mathbb{R}^n,$ and
		\item $\left|\frac{1}{p(x)}-\frac{1}{p_\infty}\right| \leq \frac{C}{\log(e+|x|)}$  for all  $x\in \mathbb{R}^n.$
	\end{enumerate}
\end{Definition}

With this information at hand, we continue with the presentation of our main results. In a first theorem, we will study the existence and uniqueness of a mild solution by considering a Lebesgue space of variable exponent in the time variable $t>0$, 
and by setting a classical $L^q$-space in the space variable $x\in\R^3$. 
This first result reads as follows. 

\begin{Theoreme}\label{Theoreme_1}
Let $\alpha\in ]\frac 1 2, 1],\ p(\cdot )\in \mPl(\Rt)$ with $p^->2$ and fix a parameter $q>\frac{3}{2\alpha -1}$ by the relationship
$\frac{\alpha}{p(\cdot)}+\frac{3}{2q}<\alpha - \frac{1}{2}$. 
If $\vf\in L^{1} \left( [0,+\infty[,  L^{q} (\R^3) \right)$ is a divergence free exterior force and if $\vu_0\in L^{q} (\Rt)$ is a divergence free initial data, 
then there exists a time $0<T<+\infty$ and a unique mild solution of the fractional Navier-Stokes equations \eqref{NS_Intro} in the space $L^{p(\cdot)} \left( [0,T], L^{q} (\R^3) \right)$.
\end{Theoreme}

We should stress the fact that, in this theorem we first measure the information in the space variable and then we measure the information in the time variable. 
In our next result, we will study a variant of the variable Lebesgue spaces for the space variable $x\in\R^3$ and we will set the classical space $L^\infty (\R^3)$
as the functional framework for the time variable $t\in \R_+$. The theorem reads as follows. 

\begin{Theoreme}\label{Theoreme_2}
Consider $\alpha\in ]\frac 1 2, 1]$, a variable exponent $p(\cdot)\in \mathcal{P}^{\log}(\mathbb{R}^3)$ such that $p^->1$, an initial data $\vu_0\in \mathcal{L}^{p(\cdot)}_{   \frac{3}{2\alpha -1}      }(\mathbb{R}^3)$ 
such that $div(\vu_0)=0$,
and let $\vf$ be a divergence free external force such that $\vf=div(\mathcal{F})$ where $\mathcal{F}$ is a tensor in  $ \mathcal{L}^{\frac{p(\cdot)}{2}}_{\frac{3}{2(2\alpha-1)}}(\mathbb{R}^3, L^\infty([0,T[))$.
If  $\|\vu_0\|_{\mathcal{L}^{p(\cdot)}_{   \frac{3}{2\alpha -1}        }} +
\|\mathcal{F}\|_{\mathcal{L}^{\frac{p(\cdot)}{2}}_{\frac{3}{2(2\alpha-1)},x}(L^\infty_t)}$ is small enough, then the fractional Navier-Stokes equations (\ref{NS_Intro}) admits a unique, global mild solution in the space 
$\mathcal{L}^{p(\cdot)}_{   \frac{3}{2\alpha -1}      }
(\mathbb{R}^3,L^\infty([0,T[))$.
\end{Theoreme}
Note that, in this theorem we have considered the \emph{mixed variable Lebesgue spaces} $\mathcal{L}^{p(\cdot)}_{\mathfrak{p}}$ (see Subsection \ref{Secc_Notaciones_Presentaciones} for the definition of these spaces) by merely technical reasons and it is motivated by the lack of flexibility in the parameters that intervene in the boundedness of the Riesz transforms involved in its proof. See Subsection \ref{Secc_Notaciones_Presentaciones} and  Remark \ref{Rem_Riesz_MixedLebesgue} below for more details on this particular issue.\\

We finish this section by remarking the fact that the  Theorems \ref{Theoreme_1} and \ref{Theoreme_2}  recently presented, extends the results obtained in \cite{MR4719441} to the case of the fractional Navier-Stokes equations \eqref{NS_Intro}. In fact, by considering $\alpha=1$, we recover the classical Navier-Stokes equations and thus the results in the mentioned paper.

\subsection*{Organization of the paper}
The present paper is structured as follows.
In Section \ref{Secc_resentaciones} we present a concise review of the main definitions and properties of the  Variable Lebesgue spaces $L^{p(\cdot)}$ and  fractional heats kernels. Section \ref{Secc_Proof_Existence} is devoted to the proof of the theorems.

\section{Preliminaries}\label{Secc_resentaciones}

To keep this paper reasonably self-contained, several results and definitions on variable Lebesgue spaces and fractional heat kernels are recalled. 
Thus, we begin by presenting a brief summary on the key elements of variable Lebesgue spaces theory. 
\subsection{Variable Lebesgue spaces}\label{Secc_Notaciones_Presentaciones}
We start this subsection by considering some basic conventions. 
For the sake of simplicity in the rest of the paper we will assume $1<p^-\leq p^+<+\infty$.
On the other hand, in order of differentiate between variable and constant exponents, we will always denote variable exponent by $p(\cdot)$.\\

Note that, the spaces $L^{p(\cdot)}(\mathbb{R}^n)$  are  Banach function spaces and they have very interesting features. 
Thus, we start by presenting the generalization of the H\"older inequalities in this setting. 
\begin{Lemme}
    Let $p(\cdot),\,q(\cdot),\,r(\cdot)\in \mathcal{P}(\mathbb{R}^n)$ be functions such that we have the pointwise relationship
$\frac{1}{p(x)}=\frac{1}{q(x)}+\frac{1}{r(x)}$, $x\in \mathbb{R}^n$. Then there exists a constant $C>0$ such that for all $f\in L^{q(\cdot)}(\mathbb{R}^n) $ and $g \in L^{r(\cdot)}(\mathbb{R}^n)$, the pointwise product $fg$ belongs to the space $L^{p(\cdot)}(\mathbb{R}^n)$ and we have the estimate
\begin{equation}\label{Holder_LebesgueVar}
\|f g\|_{L^{p(\cdot)}} \leq C\|f\|_{L^{q(\cdot)}}\|g\|_{L^{r(\cdot)}}.
 \end{equation}
\end{Lemme}
As is natural, this estimate can be easily generalized to vector fields $\vf, \vg :\mathbb{R}^n\longrightarrow \mathbb{R}^n$ and to the product $\vf\cdot \vg$. For a  proof of this result we recommend to the interested reader to \cite[Section 2.4]{Cruz_Libro} or \cite[Section 3.2]{Diening_Libro}.\\

Note that the quantity $\|\cdot\|_{L^{p(\cdot)}}$ satisfies the \emph{Norm conjugate formula} given in \cite[Corollary 3.2.14]{Diening_Libro}:

\begin{Proposition}
    Let $p(\cdot),\,q(\cdot),\,r(\cdot)\in \mathcal{P}(\mathbb{R}^n)$ be functions such that we have the pointwise relationship
$1=\frac{1}{p(x)}+\frac{1}{q(x)}$, $x\in \mathbb{R}^n$. 
Then, for all  $f\in L^{p(\cdot)}$ we have
\begin{equation}\label{Norm_conjugate_formula}
\|f\|_{L^{p(\cdot)}}
\leq
\underset{\|g\|_{L^{p'(\cdot)}}\leq 1}{\sup}\int_{\mathbb{R}^n}|f(x)||g(x)|dx.
\end{equation}
\end{Proposition}

\begin{Remarque}
Note that in the previous notions the space $\mathbb{R}^n$ can be replaced by an interval $[0,T]$. 
\end{Remarque}
To continue  we present the following embedding result \cite[Corollary 2.48]{Cruz_Libro}.
\begin{Lemme}\label{lemme_embeding}
Let consider $n\geq 1$, a bounded domain $\mathcal{X}\subset \mathbb{R}^n$ and two variable exponents  $q_1(\cdot),  q_2(\cdot)\in \mathcal{P}(\mathcal{X})$ such that $1< q_1^+, \ q_2^+ <+\infty$.  Then, $L^{q_2(\cdot)}(\mathcal{X})\subset L^{q_1(\cdot)}(\mathcal{X})$ if and only if $q_1(x)\leq q_2(x)$ almost everywhere. Furthermore, in this case we have that  
$$\|f\|_{L^{q_1(\cdot)} } \leq\left(1+\left|  \mathcal{X}\right|\right)\|f\|_{L^{q_2(\cdot)}}.$$
\end{Lemme}
It is important to emphasize that the convolution product $f\ast g$ is not well adapted to the structure of the $L^{p(\cdot)}$ spaces, in particular the Young inequalities for convolution are not valid anymore (see \cite[Section 3.6]{Diening_Libro}) and thus many of the usual operators that appear in PDEs must be treated very carefully. Note also that Fourier-based methods are not so easy to use as we lack of an alternative for the Plancherel formula.\\

In order to study the boudedness of such operators is classical to impose some conditions on the variable exponents $p(\cdot)$. Thus, 
with the condition $p(\cdot)\in \mathcal{P}^{log}(\mathbb{R}^n)$, we can state  the following definition and theorem.

\begin{Definition}
Let $f:\mathbb{R}^n\longrightarrow \mathbb{R}$ a locally integrable function. The Hardy-Littlewood maximal function $\mathcal{M}$ is defined by 
$$\displaystyle{\mathcal{M}(f)(x)=\underset{B \ni x}{\sup } \;\frac{1}{|B|}\int_{B }|f(y)|dy}$$ where $B$ is an open ball of $\mathbb{R}^n$.
\end{Definition}

\begin{Theoreme}\label{MaximalFunc_LebesgueVar}
    Let  $p(\cdot)\in \mathcal{P}^{log}(\mathbb{R}^n)$ with $p^->1$. Then, there exists a constant $C>0$ such that
\begin{equation}
\|\mathcal{M}(f)\|_{L^{p(\cdot)}}\leq C \|f\|_{L^{p(\cdot)}}.
\end{equation}
\end{Theoreme}
A proof of this result can be consulted in \cite[Section 4.3]{Diening_Libro}. At his point,  we recall a classical result about the Hardy-Littlewood maximal function (see \cite[Section 2.1]{grafakos2008classical}):
\begin{Lemme}\label{lemme_conv_maximal}
If $\varphi$ is a radially decreasing function on $\mathbb{R}^3$ and $\vf$ is a locally integrable function, then 
\begin{equation*}
|(\varphi\ast\vf)(x)| \leq \|\varphi\|_{L^1}  \mathcal{M} (\vf)(x),
\end{equation*}
where $\mathcal{M}$ is the Hardy-Littlewood maximal function.
\end{Lemme}
We remark that the usual Riesz transforms $(\mathcal{R}_j)_{1\leq j\leq n}$ defined formally in the Fourier level by $\widehat{\mathcal{R}_j(f)}(\xi)=-\frac{i\xi_j}{|\xi|}\widehat{f}(\xi)$ are also bounded in Lebesgue spaces of variable exponent. This fact is stated in the following Lemma.
\begin{Lemme}
Let $p(\cdot)\in \mathcal{P}^{log}(\mathbb{R}^n)$ and $1<p^-\leq p^+<+\infty$. Then, given $f\in L^{p(\cdot)}$, there exists a constant $C>0$ such that
\begin{equation}\label{Riesz_LebesgueVar}
\|\mathcal{R}_j(f)\|_{L^{p(\cdot)}}\leq C \|f\|_{L^{p(\cdot)}},
\end{equation}
\end{Lemme}
A proof of this result can be consulted in \cite[Sections 6.3 and 12.4]{Diening_Libro}. To continue, we introduce the following classical operator. 
\begin{Definition}\label{Definition_RieszPotential}
Let $0<\beta<n$.  Then, given a measurable function $f$, we define the Riesz potential operator $\mathcal{I}_\beta (f):\R^n \to [0,+\infty]$  by 
\begin{equation}\label{eq.Definition_RieszPotential}
\mathcal{I}_\beta(f)(x):=\int_{\mathbb{R}^n}\frac{|f(y)|}{|x-y|^{n-\beta}}dy. 
\end{equation}
\end{Definition}

The Riesz potential operator is bounded in variable Lebesgue spaces if we consider appropriate variable exponents in the class $\mathcal{P}^{log}(\mathbb{R}^n)$. A precise statement of this result reads as follows. 

\begin{Theoreme}\label{theo.PotentialRieszVariable0}
Let  $p(\cdot)\in \mathcal{P}^{log}(\mathbb{R}^n)$ and $0<\beta<n/p^+$. Then, there exists $C>0$ such that 
\begin{equation}\label{PotentialRieszVariable0}
\|\mathcal{I}_\beta(f)\|_{L^{q(\cdot)}}\leq C\|f\|_{L^{p(\cdot)}},\qquad \mbox{with} \quad  \frac{1}{q(\cdot)}=\frac{1}{p(\cdot)}-\frac{\beta}{n}.
\end{equation}
\end{Theoreme}
A proof of this theorem  can be consulted in
\cite[Section 6.1]{Diening_Libro}.

Note that the estimate \eqref{PotentialRieszVariable0} introduces a very strong relationship between the variable exponents $p(\cdot)$ and $q(\cdot)$. Thus, in order  to obtain  more flexibility on these parameters (see Remark \ref{Rem_Riesz_MixedLebesgue} below) we will consider the following spaces introduced in \cite{chamorro_paper_lpvar}.

\begin{Definition}
Let $p(\cdot)\in \mathcal{P}^{\log}(\mathbb{R}^n)$ a variable exponent and $1<\mathfrak{p}<+\infty$  a constant. Then, the mixed Lebesgue space $\mathcal{L}^{p(\cdot)}_\mathfrak{p}(\mathbb{R}^n)$ is defined by  $$\mathcal{L}^{p(\cdot)}_\mathfrak{p}(\mathbb{R}^n)=L^{p(\cdot)}(\mathbb{R}^n)\cap L^{\mathfrak{p}}(\mathbb{R}^n),$$ which can be normed 
by the quantity 
\begin{equation}\label{MixedLebesgue}
\|\cdot\|_{\mathcal{L}^{p(\cdot)}_\mathfrak{p}}=\max\{\|\cdot\|_{L^{p(\cdot)}}, \|\cdot\|_{L^{\mathfrak{p}}}\}.
\end{equation}
\end{Definition}

With the help of these spaces we have the following result. 
\begin{Proposition}\label{Proposition_RieszPotential}
Let $1<\mathfrak{p}<+\infty$ be a constant exponent, $p(\cdot)\in \mathcal{P}^{\log}(\mathbb{R}^n)$ a variable exponent and fix a parameter $0<\beta<\min\{n/p^+, n/\mathfrak{p}\}$. Given, $f\in \mathcal{L}^{p(\cdot)}_\mathfrak{p}(\mathbb{R}^n)$ and  a function $\rho(\cdot)$ satisfying  the following condition
\begin{equation}\label{Inegalite2Condition}
\rho(\cdot)=\frac{np(\cdot)}{n-s\mathfrak{p}},
\end{equation} then, there exists a constant $C>0$  such that 
\begin{equation}\label{Inegalite3}
\|\mathcal{I}_\beta(f)\|_{L^{\rho(\cdot)}}\leq C \|f\|_{\mathcal{L}^{p(\cdot)}_\mathfrak{p}}.
\end{equation}
\end{Proposition}
A proof of this result can be consulted in \cite{chamorro_paper_lpvar}.

\begin{Remarque} 
 We emphasize in particular the fact that the index $\mathfrak{p}$ is not to related to $p^-$ or $p^+$ nor to $p(\cdot)$ and this inequality gives more flexibility in the exponents than the conditions stated in Theorem  \ref{theo.PotentialRieszVariable0}.
\end{Remarque}

\begin{Remarque}\label{Rem_Holder_Mixed_Lebesgue_Var}
Note that, by construction, the mixed spaces $\mathcal{L}^{p(\cdot)}_\mathfrak{p}$ inherit the properties of the spaces $L^{p(\cdot)}$ and $L^{\mathfrak{p}}$. In particular we have the H\"older inequality $\|\varphi\|_{\mathcal{L}^{p(\cdot)}_\mathfrak{p}}\leq \|\varphi\|_{\mathcal{L}^{q(\cdot)}_\mathfrak{q}}\|\varphi\|_{\mathcal{L}^{r(\cdot)}_\mathfrak{r}}$ with $\frac{1}{p(\cdot)}=\frac{1}{q(\cdot)}+\frac{1}{r(\cdot)}$ and $\frac{1}{\mathfrak{p}}=\frac{1}{\mathfrak{q}}+\frac{1}{\mathfrak{r}}$ and of course the Riesz transforms are also bounded in these spaces.\\
\end{Remarque}
For more details about the Lebesgue spaces of variable exponent, their inner structure as well as many other properties, see the books \cite{Cruz_Libro},  \cite{Diening_Libro} and \cite{BookALEX}.\\

\subsection{Some preliminary estimates on the  kernels of $e^{-t(-\Delta)^{\frac{\alpha}{2}}}$ and $e^{-t(-\Delta)^{\frac{\alpha}{2}}}\mathbb{P} \operatorname{div}(\cdot)$ }

We start this subsection we recall some estimates related to the fractional heat kernel involved in the integral formulation of the system  \eqref{NS_Intro}.

\begin{Remarque}\label{remark 2.1 MIAO}
Consider $x\in \R^n$ and the fractional heat kernel $\mathfrak{g}_t^\alpha (x)$ associated to $e^{-t(-\Delta)^{\frac{\alpha}{2}}}$.
Then, there exists a constant $C>0$ such that the following estimate follows
\begin{equation}\label{eq 1 1 24}
|\nabla  \mathfrak{g}_t^\alpha (x)| 
\leq
C
\dfrac{1}{(t^{\frac{1}{2\alpha}}+|x|)^{n+1}}
.
\end{equation}
\end{Remarque}
More details about this remark can be consulted in  \cite[Remark 2.1]{miao2008well}. Another useful result for our purposes is the following. 

\begin{Lemme}\label{Lemma 3.1 MIAO}
Consider $x\in \R^n$, the fractional heat kernel $\mathfrak{g}_t^\alpha (x)$,  and  the parameters $1\leq r \leq p \leq +\infty $. Then, for $\alpha,\nu >0$,   there exists $C>0$ such that 
\begin{enumerate}
    \item  $\left\|
    \mathfrak{g}_t^\alpha *
    \varphi(x)\right\|_{L^p}
\leq
C t^{-\frac{n}{2 \alpha}\left(\frac{1}{r}-\frac{1}{p}\right)}\|\varphi\|_{L^r}$,
\item $\left\|(-\triangle)^{\nu / 2} \mathfrak{g}_t^\alpha * 
\varphi(x)\right\|_{L^p} \leq C t^{-\frac{v}{2 \alpha}-\frac{n}{2 \alpha}\left(\frac{1}{r}-\frac{1}{p}\right)}\|\varphi\|_{L^r}$.
\end{enumerate}
\end{Lemme}
A proof of this result can be consulted in \cite[Lemma 3.1]{miao2008well}.
Finally, we recall an useful estimate related to the kernel of $e^{-t(-\Delta)^{\frac{\alpha}{2}}}\mathbb{P} \operatorname{div}(\cdot)$.
\begin{Remarque}\label{remark_kernek_K}
	Consider $x\in \R^n$ and the  kernel $K_t^\alpha (x)$ associated to $e^{-t(-\Delta)^{\frac{\alpha}{2}}}\mathbb{P} \operatorname{div}(\cdot)$.
	Then, there exists a constant $C>0$ such that the following pointwise estimate follows
\begin{equation}\label{EstimateKernel}
	|  {K}_t^\alpha (x)| 
	\leq
	C
	\dfrac{1}{(t^{\frac{1}{2\alpha}}+|x|)^{n+1}}.
\end{equation}
\end{Remarque}
A detailed treatment of this result can be consulted in \cite{qian2023asymptotic}. In particular we recomend to the interested reader to see the discussion around the equation (2.9) in the mentioned paper. 

\section{Mild solutions in variable Lebesgue spaces}\label{Secc_Proof_Existence}
We present here a  general approach to mild solutions for the fractional Navier-Stokes equations (\ref{NS_Intro}) in the setting of variable Lebesgue spaces. These mild solutions are obtained via the following classical result:
\begin{Theoreme}[Banach-Picard  principle]\label{BP_principle}
Consider a Banach space $(E,\|\cdot \|_E )$ and a bounded bilinear application $\mathcal{B}: E \times E \longrightarrow E$: 
$$\|\mathcal{B}(e,e)\|_{E}\leq C_{\mathcal{B}}\|e\|_E\|e\|_E.$$
Given $e_0\in E$  such that $\|e_0\|_E \leq \delta$ with $0<\delta < \frac{1}{4C_{\mathcal{B}}}$, then the equation 
$$e = e_0 -  \mathcal{B}(e,e),$$
admits a unique solution $e\in E$ which satisfies $\| e \|_E \leq 2 \delta$.
\end{Theoreme}

In order to apply this result to the fractional Navier-Stokes equations (\ref{NS_Intro}) we need to get rid of the pressure $P$ and for this we apply to this system the Leray projector $\mathbb{P}$ defined by
$\mathbb{P}:=\text{Id} +  \vn (-\Delta)^{-1} div$. 
On the other hand,
note that the divergence-free condition allows us to recover the pressure $P$ from $\vu$ through the identity 
\begin{equation*}
    - \Delta P = div( \vu \cdot \vn \vu ), 
\end{equation*}
thus, in the following we will focus our analysis on the velocity vector field $\vu$.

\begin{Remarque}
We emphasise the fact that
the Leray projector $\mathbb{P}$ 
can be defined equivalently by considering  Riesz transforms. In fact, given a function $\vec{g}$ in a functional space $\mathcal{X}$, the Leray projector is defined as follows:   $\mathbb{P}(\vec{g})=(Id_{3\times 3}-\vec{R}\otimes\vec{R})(\vec{g})$ where $\vec{R}=(R_1, R_2, R_3)$, with $R_j$ denoting the $j$-th Riesz transform. Thus, if the Riesz transforms are bounded in $\mathcal{X}$, then the Leray's projector $\mathbb{P}$ is also bounded on $\mathcal{X}$, and then we can write  $\|\mathbb{P}(\vec{g})\|_\mathcal{X}\leq C\|\vec{g}\|_\mathcal{X}$.
\end{Remarque}

As we anticipated above, the pressure term $P$ involved in \eqref{NS_Intro} has the property   $\mathbb{P}(\vn P)\equiv 0$. 
Moreover, given  a divergence free vector field,  we have the identity $\mathbb{P}(\vec{g})= \vec{g}$. A proof of these facts and many others properties of the projector $\mathbb{P}$, can be consulted in the book  \cite{Lemarie2016}.\\

Since $\vu$ and $\vf$ are divergence free, by applying  Leray's projector $\mathbb{P}$ to \eqref{NS_Intro} we obtain the equation
$$
\begin{cases}
\partial_t\vu=-(-\Delta)^{\alpha}\vu-\mathbb{P}(div(\vu \otimes \vu))+\vf, \quad div(\vu)=0,\\[2mm]
\vu(0,x)=\vu_0(x), \qquad x\in \mathbb{R}^3.
\end{cases}
$$
Now, due to the Dumahel formula, we can write this equation in the following form 
\begin{equation}\label{NS_Integral}
\vu(t,x)=
\mathfrak{g}^\alpha_t\ast \vu_0(x)
+
\int_{0}^t
\mathfrak{g}^\alpha_{t-s}\ast\vf(s, x)
ds
-
\int_{0}^t
\mathfrak{g}^\alpha_{t-s}\ast \mathbb{P}(div(\vu \otimes \vu))(s, x)
ds,
\end{equation}
where $\mathfrak{g}^\alpha_t$ is the usual fractional heat kernel. 
In the following we will consider  the integral equation above in order to apply the Banach-Picard principle. To this end,  we set the bilinear application as 
\begin{equation}\label{bilinear-appl}
   \mathcal{B}(\vu,\vv) =
\int_{0}^t
\mathfrak{g}^\alpha_{t-s}\ast \mathbb{P}(div(\vu \otimes \vv))(s, x)ds
, 
\end{equation} 
and we
consider the term $\mathfrak{g}^\alpha_t\ast \vu_0(x)
+
\int_{0}^t
\mathfrak{g}^\alpha_{t-s}\ast\vf(s, x)
ds$ as $e_0$ (in Theorem \ref{BP_principle}).

\begin{Remarque}\label{RMK_kernel_and_B}
Note that, given  the kernel $K_t^\alpha (x)$ of the convolution operator $e^{-t(-\Delta)^{\frac{\alpha}{2}}}\mathbb{P} \operatorname{div}(\cdot)$, the $j$-th component of
the bilinear application defined in \eqref{bilinear-appl}
can be written  as
\begin{equation*}
	\mathcal{B}_j(\vu,\vu) = \int_{0}^{t} \sum_{h,k =1}^3 (K^\alpha_{t-s})_{j} * (\vu_h\vu_k)(s)ds.
\end{equation*}
\end{Remarque}
A detailed treatment of this result and the kernel $K_t^\alpha (x)$ can be consulted in \cite[Section 2]{qian2023asymptotic}.

\subsection{Proof of Theorem \ref{Theoreme_1}}
In the following we consider a variable exponent $p(\cdot)\in \mathcal{P}^{log}([0,+\infty[)$, and the functional space 
$$
\mathfrak{E}_T=
 L^{p(\cdot)} \left( [0,T], L^{q} (\Rt) \right),$$  
with $T\in ]0,+\infty[ $ to be precised later. 
The space $\mathfrak{E}_T$ is endowed with a Luxemburg norm as follows
\begin{equation}\label{Norm_PointFixe_2}
\| \vec{\varphi}\|_{\mathfrak{E}_T}=\inf\left\{\lambda > 0: \,\int_0^{T}\left|\frac{ \|
\vec{\varphi} (t,\cdot)\|_{L^{q}}}{\lambda}\right|^{p(t)} dt \leq1\right\}.
\end{equation}
Under this functional setting we will consider the Banach-Picard principle to construct mild solutions for the integral equation  (\ref{NS_Integral}).
More precisely, in the following we will prove 3 propositions which will provide the core of the hypotheses of Theorem \ref{BP_principle}.
\\

To this end, we start by proving the following result regarding a control to the initial data. 

\begin{Proposition}\label{prop.Control_Uo_Lplq}
Let $\alpha\in ]\frac 1 2, 1],\ p(\cdot )\in \mPl(\Rt)$ with $p^->2$ and fix an index $q>\frac{3}{2\alpha -1}$ by the relationship
$\frac{\alpha}{p(\cdot)}+\frac{3}{2q}<\alpha - \frac{1}{2}$. 
Consider a function
$\vu_0\in L^{q} (\Rt)$. Then, there exists a constant $C_1>0$ such that 
\begin{equation}\label{Control_Uo_Lplq}
\|\mathfrak{g}_t^\alpha * \vu_0  \|_{ \mathfrak{E}_T}\leq C_1\| \vu_0  \|_{L^{q}}.
\end{equation} 
\end{Proposition}
\begin{proof}
In order to conclude
the inequality \eqref{Control_Uo_Lplq}, we  start by considering a Young inequality to obtain 
$$\|\mathfrak{g}_t^\alpha * \vu_0  \|_{L^{q} (\Rt)} \leq \|\mathfrak{g}_t^\alpha   \|_{L^{1} (\Rt)}\| \vu_0  \|_{L^{q} (\Rt)} = \| \vu_0  \|_{L^{q} (\Rt)}.$$ 
Before continue, let consider the following classical result in the context of Variable Lebesgue spaces (see\cite[Lemma 3.2.12, Section 3.2]{Diening_Libro}).
\begin{Lemme}\label{Lemma_subset}
Let $p(\cdot)\in \mathcal{P}([0,+\infty[)$ such that $1<p^-\leq p^+<+\infty$. Then, there exists $C>0$ such that 
$$
\|1\|_{L^{p(\cdot)}([0,T])} \leq
C \max\left\{T^{\frac{1}{p^{-}}}, T^{\frac{1}{p^{+}}}\right\}.$$
\end{Lemme}
Thus, by taking the $L^{p(\cdot)}_t$-norm  we obtain 
\begin{eqnarray}\label{Estimation_DonneeInitiale}
\|\mathfrak{g}_t^\alpha * \vu_0  \|_{L^{p(\cdot)} _tL^{q}_x }
&\leq& C\|\vu_0\|_{L^{q} (\Rt)}\|1\|_{L^{p(\cdot)}([0,T])} 
\\
&\leq& C \|\vu_0\|_{L^{q} (\Rt)}\max\left\{T^{\frac{1}{p^{-}}}, T^{\frac{1}{p^{+}}}\right\}.
\end{eqnarray}
Thus, by considering $C_1=C\max\left\{T^{\frac{1}{p^{-}}}, T^{\frac{1}{p^{+}}}\right\}$ we deduce \eqref{Control_Uo_Lplq} and we conclude the proof of Proposition \ref{prop.Control_Uo_Lplq}. 
\end{proof}
\begin{Proposition}\label{prop.Control_force_LpLq}
Let $\alpha\in ]\frac 1 2, 1],\ p(\cdot )\in \mPl(\Rt)$ with $p^->2$ and fix an index $q>\frac{3}{2\alpha -1}$ by the relationship
$\frac{\alpha}{p(\cdot)}+\frac{3}{2q}<\alpha - \frac{1}{2}$. 
Then, given a function
 $\vf\in L^{p(\cdot)} \left( [0,+\infty[,  L^{q} (\R^3) \right)$, there exists a numerical constant $C_2>0$ such that
\begin{equation} \label{Control_force_LpLq}
\left\|\int_0^t\mathfrak{g}_{t-s}^\alpha *\vf(s, \cdot) ds\right\|_{\mathfrak{E}_T} \leq C_2\|\vf \|_{L^1_t(L^q_x)}. 
\end{equation} 
\end{Proposition}
\begin{proof}
To conclude the inequality \eqref{Control_force_LpLq},  we start by considering
the $L^q_x$-norm to the term $\int_0^t\mathfrak{g}_{t-s}^\alpha \ast \vf (s, \cdot)ds$ 
in order to get 
\begin{eqnarray*}
\left\|
\int_0^t\mathfrak{g}_{t-s}^\alpha \ast \vf (s, \cdot)ds
\right\|_{L^{q}} 
&\leq&  \int_0^t\|\mathfrak{g}_{t-s}^\alpha\|_{L^1} \| \vf(s, \cdot) \|_{L^{q}}ds
\\
&\leq& C\| \vf\|_{L^1_t(L^{q}_x)}.
\end{eqnarray*}
Then, by taking the $L^{p(\cdot)}_t$-norm in the time variable and considering Lemma \ref{Lemma_subset} we deduce the estimates
\begin{eqnarray}
\left\|\int_0^t\mathfrak{g}_{t-s}^\alpha \ast \vf(s, \cdot) ds\right\|_{L^{p(\cdot)}_t(L^{q}_x)}
&\leq &
C\left\|\| \vf  \|_{L^1_t(L^{q}_x)}\right\|_{L^{p(\cdot)}_t}
\\
&\leq&
C\| \vf\|_{L^1_t(L^{q}_x)}\|1\|_{L^{p(\cdot)}([0,T])}\notag\\
&\leq &C\| \vf\|_{L^1_t(L^{q}_x)}\max\left\{T^{\frac{1}{p^{-}}}, T^{\frac{1}{p^{+}}}\right\}.\label{Estimation_ForceExterieure}
\end{eqnarray}
Thus, by considering $C_2=C\max\left\{T^{\frac{1}{p^{-}}}, T^{\frac{1}{p^{+}}}\right\}$ we deduce the inequality \eqref{Control_force_LpLq} and then we conclude the proof of Proposition \ref{prop.Control_force_LpLq}. 
\end{proof}

\begin{Proposition}\label{prop.Control_Nolineal}
Let $\alpha\in ]\frac 1 2, 1],\ p(\cdot )\in \mPl(\Rt)$ with $p^->2$ and fix an index $q>\frac{3}{2\alpha -1}$ by the relationship
$\frac{\alpha}{p(\cdot)}+\frac{3}{2q}<\alpha - \frac{1}{2}$. 
Then, 
there exists a constant $C_{\mathcal{B}}>0$ such that
\begin{equation}\label{Control_Nolineal}
\left\|
\int_{0}^t\mathfrak{g}_{t-s}^\alpha\ast \mathbb{P}(div(\vu \otimes \vu))(s, \cdot)ds
\right\|_{\mathfrak{E}_T}\leq C_{\mathcal{B}}\|\vu\|_{\mathfrak{E}_T}\|\vu\|_{\mathfrak{E}_T}.
\end{equation}
\end{Proposition}
\begin{proof}

To conclude the inequality \eqref{Control_Nolineal},  we start by considering
the $L^q_x$-norm to the term $\int_{0}^t\mathfrak{g}_{t-s}^\alpha\ast \mathbb{P}(div(\vu \otimes \vu))(s, \cdot)ds$ 
to get  
$$\left\|
\int_{0}^t\mathfrak{g}_{t-s}^\alpha\ast \mathbb{P}(div(\vu \otimes \vu))(s, \cdot)ds
\right\|_{L^q}\leq C\int_{0}^t\left\|
\mathfrak{g}_{t-s}^\alpha
\ast \mathbb{P}(div(\vu \otimes \vu))(s, \cdot)
\right\|_{L^q}ds.$$
Considering the properties of the Leray projector $\mathbb{P}$ and the fact that the Riesz transform which define it are bounded on classical Lebesgue spaces $L^q$, with $1<q<+\infty$, we conclude 
$$\left\|\int_{0}^t\mathfrak{g}_{t-s}^\alpha\ast \mathbb{P}(div(\vu \otimes \vu))(s, \cdot)ds
\right\|_{L^q}\leq C\int_0^t \left\|\vn \mathfrak{g}_{t-s}^\alpha * \vu\otimes\vu(s, \cdot)\right\|_{L^q} ds.$$
Now, note that 
a Young inequality, with $1+\frac{1}{q}=\frac{q-1}{q}+\frac{2}{q}$, 
and Lemma \ref{Lemma 3.1 MIAO}, with parameters $\nu = 1  $, $n=3$, 
$p=p$ and $r=\frac{p}{2}$, yield 
\begin{eqnarray*}
\int_0^t \left\|\vn \mathfrak{g}_{t-s}^\alpha * \vu\otimes\vu(s, \cdot)\right\|_{L^q} ds
&\leq&
C\int_0^t  
 \frac{1}{
 (t-s)^{
 \frac{1}{2\alpha}+\frac{3}{2\alpha q}
 }
 }
\|\vu(s,\cdot)\|_{L^{q}}\|\vu(s,\cdot)\|_{L^{q}} ds.
\end{eqnarray*}
To continue, we take $L^{p(\cdot)}_t$-norm and we consider the 
 conjugate exponent $p'(\cdot)$ defined by $1=\frac{1}{p(\cdot)}+\frac{1}{p'(\cdot)}$. Then, by considering the norm conjugate formula (\ref{Norm_conjugate_formula}) we can write
\begin{eqnarray}
\left\|\int_{0}^t\mathfrak{g}_{t-s}^\alpha\ast \mathbb{P}(div(\vu \otimes \vu))(s, \cdot)ds
\right\|_{L_t^{p(\cdot)}(L_x^q)}
&\leq&
\left\|\int_0^t 
\frac{1}{
 (t-s)^{
 \frac{1}{2\alpha}+\frac{3}{2\alpha q}
 }
 }
\left\|
\vu(s,\cdot)
\right\|_{L^{q}}^2ds\right\|_{L_t^{p(\cdot)}([0,T])}\label{Identite_NormeDualite}
\\
&\leq &
\sup_{
\| \psi \|_{L^{p'(\cdot)}}\leq 1
} 
\int_0^{T} \int_0^t 
\frac{|\psi(t)|}{
|t-s|^{
\frac{1}{2\alpha}+\frac{3}{2 \alpha q}
}}
\left\|\vu(s,\cdot)\right\|_{L^{q}}^2ds \,dt.\notag
\end{eqnarray}
The Fubini Theorem yields 
$$
\sup_{\| \psi \|_{L^{p'(\cdot)}}\leq 1}
\int_0^{T} \int_0^t 
\frac{|\psi(t)|}{
|t-s|^{
\frac{1}{2\alpha}+\frac{3}{2  \alpha    q
}
}}\left\|\vu(s,\cdot)\right\|_{L^{q}}^2ds \,dt
=
\sup_{\| \psi \|_{L^{p'(\cdot)}}\leq 1}
\int_0^T\int_0^T
\frac{1_{ \{ 0<s<t \} } |\psi(t)| }{|t-s|^{
\frac{1}{2\alpha}+\frac{3}{2  \alpha    q}
}}
dt\|\vu(s,\cdot)
\|^2_{L^q}ds,
$$
thus, by extending the function $\psi(t)$ by zero 
on $\mathbb{R} \setminus [0,T]$, we obtain 
\begin{eqnarray*}
\sup_{\| \psi \|_{L^{p'(\cdot)}}\leq 1}
\int_0^{T} \int_0^t 
\frac{|\psi(t)|}{|t-s|^{
\frac{1}{2\alpha}+\frac{3}{2  \alpha    q}
}}\left\|\vu(s,\cdot)\right\|_{L^{q}}^2ds \,dt
&=&
\sup_{\| \psi \|_{L^{p'(\cdot)}}\leq 1}
\int_0^T
\left(\int_{-\infty}^{+\infty}\frac{|\psi(t)| }{
|t-s|^{
\frac{1}{2\alpha}+\frac{3}{2  \alpha    q}
}
}
dt
\right)
\|\vu(s,\cdot)
\|^2_{L^q}ds
\\
&=&\sup_{\| \psi \|_{L^{p'(\cdot)}}\leq 1}\int_0^T\mathcal{I}_\beta(|\psi|)(s)\|\vu(s,\cdot)\|^2_{L^q}ds,
\end{eqnarray*}
where $\mathcal{I}_\beta$ is the 1D Riesz potential   with $\beta=1- \frac{1}{2\alpha}-\frac{3}{2  \alpha    q}
<1$ (see Definition \ref{Definition_RieszPotential}). 
\begin{Remarque}
We emphasise the fact that the constraints  $\frac{3}{2\alpha -1}<q$ and $\alpha \in ] \frac 1 2, 1]$ imply $0<1- \frac{1}{2\alpha}-\frac{3}{2  \alpha    q}<1$, and thus the Riesz potential considered is well defined.
\end{Remarque}
A H\"older inequality with $1=\frac{1}{p(\cdot)}+\frac{1}{p(\cdot)}+\frac{1}{ \tilde{p}(\cdot)}$  yields
$$\sup_{\| \psi \|_{L^{p'(\cdot)}}\leq 1}\int_0^T\mathcal{I}_\beta(|\psi|)(s)\|\vu(s,\cdot)\|^2_{L^q}ds\leq C\sup_{\| \psi \|_{L^{p'(\cdot)}}\leq 1}\|\mathcal{I}_\beta(|\psi|)\|_{L_t^{\tilde{p}(\cdot)}}\Big\|\|\vu(\cdot,\cdot)\|_{L_x^q}\Big\|_{L_t^{p(\cdot)}}\Big\|\|\vu(\cdot,\cdot)\|_{L_x^q}\Big\|_{L_t^{p(\cdot)}}.$$
\begin{Remarque} 
Note that the condition $p^->2$ in the statement of the proposition become from the  relationship $1=\frac{2}{p(\cdot)}+\frac{1}{ \tilde{p}(\cdot)}$ .
\end{Remarque}
Thus by considering the indexes defined by the relationship
\begin{equation}\label{Indices_Riesz}
\frac{1}{\tilde{p}(\cdot)}
=
\frac{1}{r(\cdot)}-
\left(
1- \frac{1}{2\alpha}-\frac{3}{2  \alpha    q}
\right),
\end{equation}
on Theorem \ref{theo.PotentialRieszVariable0}, we obtain 
\begin{multline}
\sup_{\| \psi \|_{L^{p'(\cdot)}}\leq 1}\|\mathcal{I}_\beta(|\psi|)\|_{L_t^{\tilde{p}(\cdot)}}\|\vu(\cdot,\cdot)\|_{L_t^{p(\cdot)}(L_x^q)}\|\vu(\cdot,\cdot)\|_{L_t^{p(\cdot)}(L_x^q)}
\\
\leq C\sup_{\| \psi \|_{L^{p'(\cdot)}}\leq 1}\|\psi\|_{L_t^{r(\cdot)}}\|\vu(\cdot,\cdot)\|_{L_t^{p(\cdot)}(L_x^q)}\|\vu(\cdot,\cdot)\|_{L_t^{p(\cdot)}(L_x^q)}.\label{Estimation_AvantInclusion}
\end{multline}
Gathering the hypothesis $\frac{\alpha}{p(\cdot)}
+
\frac{3}{2q}
<\alpha - \frac{1}{2}$
 and the relationship  \eqref{Indices_Riesz} we conclude that 
$$
\frac{1}{\tilde{p}(\cdot)}=1-\frac{2}{p(\cdot)}
\quad \text{and} \quad 
\frac{1}{p'(\cdot)}=1-\frac{1}{p(\cdot)},
$$ 
and then, we deduce $r(\cdot)<p'(\cdot)$. 
Considering Lemma \ref{lemme_embeding} with $r(\cdot)<p'(\cdot)$ and $\mathcal{X}= [0,T]$ in (\ref{Estimation_AvantInclusion}) we get  
\begin{eqnarray*}
\sup_{\| \psi \|_{L^{p'(\cdot)}}\leq 1}\|\psi\|_{L_t^{r(\cdot)}}\|\vu(\cdot,\cdot)\|_{L_t^{p(\cdot)}(L_x^q)}\|\vu(\cdot,\cdot)\|_{L_t^{p(\cdot)}(L_x^q)}\qquad\qquad\qquad\qquad\qquad\qquad\\
\leq \sup_{\| \psi \|_{L^{p'(\cdot)}}\leq 1}(1+T)\|\psi\|_{L_t^{p'(\cdot)}}\|\vu(\cdot,\cdot)\|_{L_t^{p(\cdot)}(L_x^q)}\|\vu(\cdot,\cdot)\|_{L_t^{p(\cdot)}(L_x^q)}\\
\leq (1+T)\|\vu(\cdot,\cdot)\|_{L_t^{p(\cdot)}(L_x^q)}\|\vu(\cdot,\cdot)\|_{L_t^{p(\cdot)}(L_x^q)}.
\end{eqnarray*}
Thus, by gathering these estimates with (\ref{Identite_NormeDualite}) we deduce 
\begin{equation}\label{Estimation_ApplicationBilineaire}
\left\|\int_{0}^t\mathfrak{g}_{t-s}^\alpha\ast \mathbb{P}(div(\vu \otimes \vu))(\cdot, \cdot)ds
\right\|_{L_t^{p(\cdot)}(L_x^q)}\leq C(1+T)\|\vu(\cdot,\cdot)\|_{L_t^{p(\cdot)}(L_x^q)}\|\vu(\cdot,\cdot)\|_{L_t^{p(\cdot)}(L_x^q)}.
\end{equation}
Thus, by considering $C_{\mathcal{B}}=C(1+T)$ we deduce the inequality \eqref{Control_Nolineal} and then we conclude the proof of Proposition \ref{prop.Control_Nolineal}. 
\end{proof}

\subsection*{End of the proof of Theorem \ref{Theoreme_1}}
Gathering together the hypothesis assumed in the statement of the theorem and the estimates obtained in  Propositions \ref{prop.Control_Uo_Lplq}, \ref{prop.Control_force_LpLq} and \ref{prop.Control_Nolineal}, we conclude that there exists $0<T<+\infty$ such that  
$$\|\vu_0\|_{L^{q} (\Rt)}+\| \vf\|_{L^1_t(L^{q}_x)}\leq\frac{C}{(1+T)\max\left\{T^{\frac{1}{p^{-}}}, T^{\frac{1}{p^{+}}}\right\}}.$$
Thus, by applying the Banach-Picard principle we obtain the desired solution. 
With this we conclude the proof of Theorem \ref{Theoreme_1}.

\subsection{Proof of  Theorem \ref{Theoreme_2}}
In the following we consider a variable exponent $p(\cdot)\in \mathcal{P}^{log}([0,+\infty[)$, and the functional space 
$$
\mathscr{E}=
\mathcal{L}^{p(\cdot)}_{   \frac{3}{2\alpha -1}      }
(\mathbb{R}^3,L^\infty([0,T[)).$$  
The space $\mathscr{E}$ is endowed with a Luxemburg norm as follows
\begin{equation}\label{Norm_PointFixe}
\|\cdot\|_{\mathscr{E}}=\max\{\|\cdot\|_{L^{p(\cdot)}_x(L^\infty_t)}, \|\cdot\|_{
L^{   \frac{3}{2\alpha -1}        }_x(L^\infty_t)}\},
\end{equation}
Under this functional setting we will consider the Banach-Picard principle to construct mild solutions for the integral equation  (\ref{NS_Integral}).
More precisely, in the following we will prove 3 propositions which will provide the core of the hypotheses of Theorem \ref{BP_principle}.
\\

To this end, we start by proving the following result regarding a control to the initial data.

\begin{Proposition}\label{prop.Control_Uo}
Consider $\alpha\in ]\frac 1 2, 1]$,  consider a variable exponent $p(\cdot)\in \mathcal{P}^{log}(\mathbb{R}^3)$ such that $p^->1$, and 
 a divergence free function $\vu_0\in \mathcal{L}^{p(\cdot)}_{   \frac{3}{2\alpha -1}      }(\mathbb{R}^3)$. Then, there exists a constant $C>0$ such that 
 \begin{equation}\label{Control_Uo}
\|\mathfrak{g}_t^\alpha\ast \vu_0\|_{\mathscr{E}}\leq C\|\vu_0\|_{\mathcal{L}^{p(\cdot)}_{   \frac{3}{2\alpha -1}        } }.
\end{equation}
\end{Proposition}
\begin{proof}
Since $\vu_0\in L^{\frac{3}{2\alpha -1}}\subset L^1_{\text{loc}}$, we have that $\vu_0$ is a locally integrable function. Now, considering that the fractional heat kernel $\mathfrak{g}_t^\alpha$ is a  radially decreasing function, 
by  Lemma \ref{lemme_conv_maximal} we can write
$$\|\mathfrak{g}_t^\alpha\ast \vu_0(x)\|_{L^\infty_t}\leq C\mathcal{M}(\vu_0)(x).$$
Thus, by  recalling  
the norm defined in \eqref{Norm_PointFixe}, we obtain the estimate
\begin{eqnarray*}
\|\mathfrak{g}_t^\alpha\ast \vu_0\|_{\mathscr{E}}
&\leq & C\max\{\|\mathcal{M}(\vu_0)\|_{L^{p(\cdot)}}, \|\mathcal{M}(\vu_0)\|_{L^{\frac{3}{2\alpha -1}  }}\}.
\end{eqnarray*}
Now, since $p(\cdot)\in \mathcal{P}^{log}(\mathbb{R}^3)$ with $p^->1$, by Theorem \ref{MaximalFunc_LebesgueVar} we conclude that the maximal function $\mathcal{M}$ is bounded in the Lebesgue space $L^{p(\cdot)}(\mathbb{R}^3)$. 
Considering this, and the fact that  $\mathcal{M}$ is also bounded in $L^{\frac{3}{2\alpha -1}  }$,  we obtain
$$\|\mathfrak{g}_t^\alpha\ast \vu_0\|_{\mathscr{E}}\leq C\max\big\{\|\vu_0\|_{L^{p(\cdot)}},\|\vu_0\|_{L^{\frac{3}{2\alpha -1}  }}\big\}\leq  C\|\vu_0\|_{\mathcal{L}^{p(\cdot)}_{\frac{3}{2\alpha -1}  }}.$$
With this we conclude the proof. 
\end{proof}

\begin{Proposition}\label{prop.Control_f_preambulo}
Consider $\alpha\in ]\frac 1 2, 1]$, a variable exponent $p(\cdot)\in \mathcal{P}^{\log}(\mathbb{R}^3)$ such that $p^->1$,  and  a function $\mathcal{F}\in \mathcal{L}^{\frac{p(\cdot)}{2}}_{\frac{3}{2(2\alpha-1)}}(\mathbb{R}^3, L^\infty([0,T[))$.
Then, there exists a numerical constant $C>0$ such that 
\begin{equation}\label{Control_f_preambulo}
\left\|\int_{0}^t\mathfrak{g}_{t-s}^\alpha\ast div(\mathcal{F})(\cdot, \cdot)ds\right\|_{\mathscr{E}}
\leq C
\|\mathcal{F}\|_{
\mathcal{L}^{\frac{p(\cdot)}{2}}_{\frac{3}{2(2\alpha-1)},x}(L^\infty_t)
}. 
\end{equation}
\end{Proposition}
\begin{proof}  
     By the  Minkowski's integral inequality, we can write
$$\left|\int_{0}^t\mathfrak{g}_{t-s}^\alpha\ast 
div(\mathcal{F})
(s,x)ds\right|\leq C\int_{0}^t\int_{\mathbb{R}^3}|\vn\mathfrak{g}_{t-s}^\alpha(x-y)| |\mathcal{F}(s,y)|dyds.$$
Then, 
 by the decay properties of the fractional heat kernel in Remark \ref{remark 2.1 MIAO},
 and the Fubini theorem, we obtain
\begin{eqnarray*}
\left|\int_{0}^t\mathfrak{g}_{t-s}^\alpha\ast div(\mathcal{F})
(s,x)ds\right|
&\leq &C\int_{\mathbb{R}^3}\int_{0}^t
\frac{1}{  
 (\ |t-s|^{ \frac{1}{
       2\alpha }   } +
|x-y| \ )^{4}
}
|\mathcal{F}(s,y)|dsdy.
\end{eqnarray*}
Now, by considering the $L^\infty_t$ norm on $\mathcal{F}$, 
we get
$$\left|\int_{0}^t\mathfrak{g}_{t-s}^\alpha\ast 
div(\mathcal{F})
(s,x)ds\right|
\leq C\int_{\mathbb{R}^3}\int_{0}^t
\frac{1}{  
 (\ |t-s|^{ \frac{1}{
       2\alpha }   } +
|x-y| \ )^{4}
}
ds \|\mathcal{F}(\cdot,y)\|_{L^\infty_t}dy.$$
Then, considering 
the Riesz potential defined in (\ref{eq.Definition_RieszPotential}), the fact that
 $\alpha\in ]\frac 1 2, 1]$, 
and the estimate
\begin{equation}
\begin{aligned}
\int_0^t \frac{d s}{\left(|t-s|^{\frac{1}{2 \alpha}}+|x-y|\right)^{4}} & \leq \int_0^{+\infty} \frac{d s}{\left(s^{\frac{1}{2 \alpha}}+|x-y|\right)^{4}} \\
& =\int_0^{+\infty} \frac{|x-y|^{2 \alpha} d \beta}{\left(\left(|x-y|^{2 \alpha} \beta\right)^{\frac{1}{2 \alpha}}+|x-y|\right)^{4}}  =\frac{1}{|x-y|^{4-2 \alpha}} 
\int_0^{+\infty} \frac{d \beta}{\left(1+\beta^{\frac{1}{2 \alpha}}\right)^{4}},
\end{aligned}
\end{equation}
we obtain 
$$\left|\int_{0}^t\mathfrak{g}_{t-s}^\alpha\ast 
div(\mathcal{F})
(s,x)ds\right|\leq C\int_{\mathbb{R}^3}\frac{1}{|x-y|^{4-2\alpha}} \|\mathcal{F}(\cdot,y)\|_{L^\infty_t}dy=
C
\mathcal{I}_{2\alpha-1}(\|\mathcal{F}(\cdot,\cdot)\|_{L^\infty_t})(x).$$
Note that, this last estimate implies   
\begin{equation}\label{eq.24.janv}
\left\|\int_{0}^t\mathfrak{g}_{t-s}^\alpha\ast div(\mathcal{F})(s,x)ds\right\|_{L^\infty_t}\leq C\mathcal{I}_{2\alpha-1}(\|\mathcal{F}(\cdot,\cdot)\|_{L^\infty_t})(x).
\end{equation}
Then, to obtain the $\mathcal{L}^{p(\cdot)}_{\frac{3}{2\alpha-1}}$-norm given in (\ref{Norm_PointFixe}), from the estimate \eqref{eq.24.janv} we get 
\begin{eqnarray*}
\left\|\int_{0}^t\mathfrak{g}_{t-s}^\alpha\ast div(\mathcal{F})(s,x)ds\right\|_{L^{p(\cdot)}_x(L^\infty_t)}&\leq &C\left\|\mathcal{I}_{2\alpha-1}(\|\mathcal{F}(\cdot,\cdot)\|_{L^\infty_t})(\cdot)\right\|_{L^{p(\cdot)}_x},
\end{eqnarray*}
and 
\begin{eqnarray*}
\left\|\int_{0}^t\mathfrak{g}_{t-s}^\alpha\ast div(\mathcal{F})(s,x)ds\right\|_{L^{\frac{3}{2\alpha-1}}_x(L^\infty_t)}&\leq &C\left\|\mathcal{I}_{2\alpha-1}(\|\mathcal{F}(\cdot,\cdot)\|_{L^\infty_t})(\cdot)\right\|_{L^{\frac{3}{2\alpha-1}}_x}.
\end{eqnarray*}
Thus, considering  Proposition \ref{Proposition_RieszPotential} we can write 
\begin{equation}\label{eq 2 24 janv}
\left\|\mathcal{I}_{2\alpha-1}(\|\mathcal{F}(\cdot,\cdot)\|_{L^\infty_t})(\cdot)\right\|_{L^{p(\cdot)}_x}\leq C\left\|\|\mathcal{F}(\cdot,\cdot)\|_{L^\infty_t}\right\|_{\mathcal{L}^{\frac{p(\cdot)}{2}}_{\frac{3}{2(2\alpha-1)},x}}=\|\mathcal{F}\|_{\mathcal{L}^{\frac{p(\cdot)}{2}}_{\frac{3}{2(2\alpha-1)},x}(L^\infty_t)}.
\end{equation}
On the other hand, since the Riesz potentials  are bounded on the classical Lebesgue space $L^{\frac{3}{2\alpha-1}}$ we obtain 
\begin{equation}\label{eq 3 24 janv}
\left\|\mathcal{I}_{2\alpha-1}(\|\mathcal{F}(\cdot,\cdot)\|_{L^\infty_t})(\cdot)\right\|_{L^{\frac{3}{2\alpha-1}}_x}\leq C\left\|\|\mathcal{F}\|_{L^\infty_t}\right\|_{L^{\frac{3}{2(2\alpha-1)}}_x}=\|\mathcal{F}\|_{L^{\frac{3}{2(2\alpha-1)}}_x(L^\infty_t)}.
\end{equation}
Then, 
gathering together 
the norm $\|\cdot\|_{\mathscr{E}}$ given in (\ref{Norm_PointFixe})
with the estimates 
\eqref{eq 2 24 janv} and
\eqref{eq 3 24 janv}, we get 
$$\left\|\int_{0}^t\mathfrak{g}_{t-s}^\alpha\ast div(\mathcal{F})(s,x)ds\right\|_{\mathscr{E}} 
\leq C
\|\mathcal{F}\|_{
\mathcal{L}^{\frac{p(\cdot)}{2}}_{\frac{3}{2(2\alpha-1)},x}(L^\infty_t)
}
<+\infty.$$
With this we conclude the proof. 
\end{proof}

\begin{Proposition}\label{prop.Control_Nolineal2}
Consider $\alpha\in ]\frac 1 2, 1]$ and a variable exponent $p(\cdot)\in \mathcal{P}^{\log}(\mathbb{R}^3)$ such that $p^->1$. Then, 
 there exists a constant $C_{\mathcal{B}}>0$ such that 
\begin{equation}
\left\|
\int_{0}^t\mathfrak{g}_{t-s}^\alpha\ast \mathbb{P}(div(\vu \otimes \vu))(\cdot, \cdot)
ds\right\|_{\mathscr{E}}\leq C_{\mathcal{B}}\|\vu\|_{\mathscr{E}}\|\vu\|_{\mathscr{E}}.
\end{equation}
\end{Proposition}
\begin{proof}
We begin by noticing that, by considering 
the Minkowski’s integral inequality, the Remark \ref{RMK_kernel_and_B},  and the  
  decay properties of the kernel $K_t^\alpha(x)$ (see Remark \ref{remark_kernek_K}), 
    we obtain
\begin{eqnarray*}
\left|
\int_{0}^t\mathfrak{g}^\alpha_{t-s}\ast \mathbb{P}(div(\vu \otimes \vu))ds
\right|
&\leq &
C\int_{0}^t\int_{\mathbb{R}^3}
|K^\alpha_{t-s}(x-y)|
  |\vu(s,y)|  |\vu(s,y)|dyds
\\
&\leq & C\int_{\mathbb{R}^3}\int_{0}^t
\frac{1}{   (\ |t-s|^{ \frac{1}{
       2\alpha }   } +
|x-y| \ )^{4}
}
|\vu(s,y)| |\vu(s,y)|dsdy.
\end{eqnarray*}
Then, considering the $L^\infty_t$-norm we can write 
\begin{eqnarray*}
\left|
\int_{0}^t\mathfrak{g}^\alpha_{t-s}\ast \mathbb{P}(div(\vu \otimes \vu))ds
\right|
&\leq & 
C\int_{\mathbb{R}^3}
\left(\int_{0}^t 
\frac{1}{  
 (\ |t-s|^{ \frac{1}{
       2\alpha }   } +
|x-y| \ )^{4}
}
ds\right)
\|\vu(\cdot,y)\|_{L^\infty_t}
\|\vu(\cdot,y)\|_{L^\infty_t}
dy.
\end{eqnarray*}
Now, remark that similarly to in the proof of Proposition \ref{prop.Control_f_preambulo},  we have 
\begin{equation}
\begin{aligned}
\int_0^t \frac{d s}{\left(|t-s|^{\frac{1}{2 \alpha}}+|x-y|\right)^{4}} 
& = C \frac{1}{|x-y|^{4-2 \alpha}} 
\end{aligned}
\end{equation}
and then, we conclude the estimate
\begin{equation}
\begin{aligned}
\left|
\int_{0}^t\mathfrak{g}^\alpha_{t-s}\ast \mathbb{P}(div(\vu \otimes \vu))ds
\right|
&\leq & 
C
\int_{\mathbb{R}^3}
\frac{1}{|x-y|^{4-2 \alpha}} 
\|\vu(\cdot,y)\|_{L^\infty_t}
\|\vu(\cdot,y)\|_{L^\infty_t}
dy.
\end{aligned}
\end{equation}
Thus, by considering Definition 
\ref{Definition_RieszPotential}, the last inequality can be recast as 
$$\left|
\int_{0}^t\mathfrak{g}^\alpha_{t-s}\ast \mathbb{P}(div(\vu \otimes \vu))ds
\right|
\leq C
\mathcal{I}_{2\alpha -1}
\big(\|\vu\|_{L^\infty_t}\|\vu\|_{L^\infty_t}\big)
(x).$$
Now, to obtain the $\mathcal{L}^{p(\cdot)}_{\frac{3}{2\alpha-1}}$-norm given in (\ref{Norm_PointFixe}), we start by noticing that  from the estimates above we can write 
\begin{eqnarray*}
\left\|
\int_{0}^t\mathfrak{g}^\alpha_{t-s}\ast \mathbb{P}(div(\vu \otimes \vu))ds
\right\|_{L^{p(\cdot)}_x(L^\infty_t)}&
\leq&
C
\|\mathcal{I}_{2\alpha-1}
(\|\vu\|_{L^\infty_t}\|\vu\|_{L^\infty_t})\|_{L^{p(\cdot)}_x(L^\infty_t)}
,
\end{eqnarray*}
and
\begin{eqnarray*}
\left\|
\int_{0}^t\mathfrak{g}^\alpha_{t-s}\ast \mathbb{P}(div(\vu \otimes \vu))ds
\right\|_{L^{\frac{3}{2\alpha -1}}_x(L^\infty_t)}
&\leq &
C\|\mathcal{I}_{2\alpha-1}
(\|\vu\|_{L^\infty_t}\|\vu\|_{L^\infty_t})\|_{L^{\frac{3}{2\alpha -1}}_x(L^\infty_t)}.
\end{eqnarray*}
Now, considering that the Riesz potential $\mathcal{I}_{2\alpha -1}$ satisfies $\|\mathcal{I}_{2\alpha-1}(\varphi)\|_{L^{\frac{3}{2\alpha -1}}}\leq C\|\varphi\|_{L^{\frac{3}{2(2\alpha -1)}}  }$, a H\"older inequality (see Remark \ref{Rem_Holder_Mixed_Lebesgue_Var}) and Proposition \ref{Proposition_RieszPotential}, we get the estimates 
\begin{eqnarray*}
\left\|
\int_{0}^t\mathfrak{g}^\alpha_{t-s}\ast \mathbb{P}(div(\vu \otimes \vu))ds
\right\|_{L^{p(\cdot)}_x(L^\infty_t)}
&\leq & 
C \left\|\|\vu\|_{L^\infty_t}\|\vu\|_{L^\infty_t}\right\|_{\mathcal{L}^{\frac{p(\cdot)}{2}}_{   \frac{3}{2(2\alpha -1)}     }   }\leq  C\|\vu\|_{\mathcal{L}^{p(\cdot)}_{   \frac{3}{2\alpha -1}   ,x}(L^\infty_t)}
\|
\vu
\|_{\mathcal{L}^{p(\cdot)}_{   \frac{3}{2\alpha -1}      ,x}(L^\infty_t)} 
,
\end{eqnarray*}
and
\begin{eqnarray*}
\left\|
\int_{0}^t\mathfrak{g}^\alpha_{t-s}\ast \mathbb{P}(div(\vu \otimes \vu))ds
\right\|_{L^\frac{3}{2\alpha -1}_x(L^\infty_t)}
&\leq &
C\left\|\|\vu\|_{L^\infty_t}\|\vu\|_{L^\infty_t}\right\|_{L^{ \frac{3}{2(2\alpha -1)}   }_x}\leq C\|\vu\|_{L^\frac{3}{2\alpha -1}_x(L^\infty_t)}\|\vu\|_{L^\frac{3}{2\alpha -1}_x(L^\infty_t)}.
\end{eqnarray*}
\begin{Remarque}\label{Rem_Riesz_MixedLebesgue}
Note that, in the case that we had considered Theorem \ref{theo.PotentialRieszVariable0} instead of Proposition \ref{Proposition_RieszPotential}, we had obtained an estimate of the form 
$\|\mathcal{I}_{2\alpha-1}(\varphi)\|_{L^{p(\cdot)}}\leq \|\varphi\|_{L^{\frac{p(\cdot)}{2}}}$, 
which, due the strong relationship between the the variable exponents involved, yields the constant exponent
$p(\cdot)\equiv \frac{3}{2\alpha-1}$.
\end{Remarque}
Gathering together these last estimates, and using the definition of the norm $\|\cdot\|_{\mathscr{E}}$ given in (\ref{Norm_PointFixe}) we obtain 
$$\left\|
\int_{0}^t\mathfrak{g}^\alpha_{t-s}\ast \mathbb{P}(div(\vu \otimes \vu))ds
\right\|_{\mathscr{E}}\leq C_{\mathcal{B}}\|\vu\|_{\mathscr{E}}\|\vu\|_{\mathscr{E}}.$$
Thus Proposition \ref{prop.Control_Nolineal} is proven.
\end{proof}

\subsection*{End of the proof of Theorem \ref{Theoreme_2}}
Considering the hypothesis assumed and the estimates obtained in  Propositions \ref{prop.Control_Uo}, \ref{prop.Control_f_preambulo} and \ref{prop.Control_Nolineal2}, 
we get the condition 
$$C\left(
\|\vu_0\|_{\mathcal{L}^{p(\cdot)}_{   \frac{3}{2\alpha -1}        }}
+
\|\mathcal{F}\|_{\mathcal{L}^{\frac{p(\cdot)}{2}}_{\frac{3}{2(2\alpha-1)},x}(L^\infty_t)}
\right)<\frac{1}{4C_{\mathcal{B}}}.$$
Thus, by applying the Banach-Picard principle stated in Theorem \ref{BP_principle}, we obtain the desired solution. With this we
conclude the proof of Theorem \ref{Theoreme_2}.
\\

\paragraph{\bf Acknowledgements.} 
The author warmly thanks Pierre-Gilles Lemarié-Rieusset and Diego Chamorro for their helpful comments and advises. The author is supported by the ANID postdoctoral program BCH 2022 grant No. 74220003.

	
	\bibliographystyle{siam}
	\bibliography{biblioFNS}

\end{document}